\documentclass[12pt, a4wide]{amsart}
\usepackage{ifthen,amssymb,graphicx}
\usepackage[all]{xy}
\usepackage[usenames]{color}
\usepackage{mathrsfs}

\newcommand{\cdbc}[1]{{D}^b_c(#1)}

\newcommand{\fmf}[3]{{\Phi^{#1}_{{\scriptscriptstyle #2\!\rightarrow\! #3}}}}
\newcommand{\Hom}{{\operatorname{Hom}}}

\newcommand{\sCM}{{\underline{\text{CM}}}}
\newcommand{\CM}{{\text{CM}}}

\newcommand{\lotimes}{{\,\stackrel{\mathbf L}{\otimes}\,}}

\DeclareMathOperator{\Img}{{Im}}

\newcommand{\cO}{{\mathcal O}}

\newcommand{\bR}{{\mathbf R}}

\newcommand{\cplx}[1]{{{\mathcal #1}^{\scriptscriptstyle\bullet}}}

\newcommand{\dcplx}[1]{{{\mathcal #1}^{\scriptscriptstyle\bullet\vee}}}
\newcommand{\dSHom}[1]{{\mathcal{H}om_{#1}^{\scriptscriptstyle\bullet}}}

\newcommand{\iso}{{\,\stackrel {\textstyle\sim}{\to}\,}}

\newtheorem{thm}{Theorem}[section]
\newtheorem*{thm*}{Theorem}
\newtheorem{cor}[thm]{Corollary}
\newtheorem{lem}[thm]{Lemma}
\newtheorem{prop}[thm]{Proposition}

\theoremstyle{definition}

\theoremstyle{remark}
\newtheorem{rema}[thm]{Remark}
\newtheorem{exe}[thm]{Example}

\numberwithin{equation}{section} 

\begin{document}
\title{Fourier-Mukai partners of singular genus one curves}

\thanks {{\it Author's address: }Departamento de Matem\'aticas, Universidad de Salamanca, Plaza de la Merced 1-4, 37008, Salamanca, tel: +34 923294456; fax +34 923294583, anacris@usal.es.\\
Work supported by the research project MTM2009-07289 (MEC)\\
}
\date{\today}


\author[A. C. L\'opez Mart\'{\i}n]{Ana Cristina L\'opez Mart\'{\i}n}
\address{Departamento de Matem\'aticas and Instituto Universitario de F\'{\i}sica Fundamental y Matem\'aticas
(IUFFyM), Universidad de Salamanca, Plaza de la Merced 1-4, 37008
Salamanca, Spain.}

\begin{abstract}  
The objective of the paper is to prove that,  as it happens for smooth elliptic curves, any Fourier-Mukai partner of  a projective reduced Gorenstein curve of genus one and trivial dualising sheaf,  is isomorphic to itself. \end{abstract}

\maketitle
\setcounter{tocdepth}{1}


\section{Introduction} Let $X$ be a projective scheme and $\cdbc{X}$ the bounded derived category of coherent sheaves on $X$. It is an interesting problem to find all projective schemes $Y$ with   $\cdbc{X}\simeq  \cdbc{Y}$, that is, to determine the set $FM(X)$ of  projective Fourier-Mukai partners of $X$. 

The aim of this paper is to prove the following 
\begin{thm}   Let $X$ be a projective reduced connected Gorenstein curve of arithmetic genus one and trivial dualising sheaf over an algebraically closed field $k$ of characteristic zero. Then any projective Fourier-Mukai partner $Y$ of $X$ is isomorphic to $X$, that is, $FM(X)=\{X\}$.
\end{thm}

The curves appearing in the statement of the theorem were classified by Catanese in \cite{Cata82}. Namely, if $X$ is a projective reduced connected Gorenstein curve of arithmetic genus one and trivial dualising sheaf, by Proposition 1.18 in  \cite{Cata82}, it is isomorphic:

\begin{enumerate}
\item Either to a Kodaira curve (always with locally planar singularities), that is,
\begin{enumerate}
\item a smooth elliptic curve, 
\item a rational curve with one node (following Kodaira's notation, that is a  curve of type $I_1$), 
\item a rational curve with one cusp (a curve of type $I_2$),
\item  a cycle of $N$ rational smooth curves (a curve of type $I_N$) with $N\geq 2$ , 
\item two rational smooth curves forming a tacnode curve  (a curve of type $II$) or 
\item three concurrent rational smooth curves in the plane (a curve of  type $IV$).
\end{enumerate} 
\item Or to a curve consisting of $N\geq 4$ rational smooth curves meeting at a point $x$  where the tangents to the branches are linearly dependent, but any $(N-1)$ of them are independent.
\end{enumerate}

Note that, by results of Kodaira and Miranda, the curves in (1) are exactly all the possible reduced fibers  appearing in a smooth elliptic surface or in a smooth elliptic threefold. This explains why they are called Kodaira curves.

The theorem was  just known for smooth elliptic curves. In this case, it was proved by Hille and Van den Bergh in \cite{HiVdB}.  For the integral singular curves in the above list, that is, for $X$ a rational curve with one node or a cusp,  Burban and KreŸssler  study in \cite{BuKr05} the derived category $\cdbc{X}$ and its group $\text{Aut}(\cdbc{X})$ of autoequivalences, but they don't tackle the question of  Fourier-Mukai partners. Thus our contribution is to pass from the classical case of a smooth elliptic curve to the singular case generalizing the result to all  singular curves of Catanese's list.
 
 In 1998, Bridgeland computes all Fourier-Mukai partners of a smooth elliptic surface. He proves in \cite{Bri98} that the partners of relatively minimal smooth elliptic surfaces are certain relative compactified Jacobians. Some recent works \cite{ChuLo, Lo} are concerned about higher dimensional elliptic fibrations. But, for the moment there is not a similar classification for the partners of higher dimensional varieties fibered by elliptic curves. Beyond its own interest, let us mention that one of the main applications of our Theorem is that it could be useful for classifying projective Fourier-Mukai partners of any elliptic fibration, that is, a flat morphism $p\colon Z\to S$ of projective schemes whose general fiber is a smooth elliptic curve. No more condition is required on either $S$ or the total space $Z$ that could be even singular. The reason is that Proposition 2.15 in \cite{HLS08} proves that a relative integral functor between the derived categories of two projective fibrations is an equivalence if and only if its restriction to every fiber (not only the smooth ones) is an equivalence. Then, if $q\colon W\to S$ is a relative Fourier-Mukai partner of $p\colon Z\to S$, for every $s\in S$ the fibers $Z_s$ and $W_s$ are Fourier-Mukai partners as well. Thus for example, by Miranda's result, to classify the partners of an elliptic threefold, it would be helpful to know about the partners of all the curves of (1) in Catanese's  list. On the other hand, the structure of some relative compactified Jacobians for certain elliptic fibrations is already known (see \cite{LM05a}).
 
To finish, one should point out that elliptic fibrations have been used is string theory, notably in connection with mirror symmetry in Calabi-Yau manifolds and $D$-branes (see \cite{BBH08} for a good survey). Some of the classic examples of families of Calabi-Yau manifolds for which there is a description of the mirror family are elliptic fibrations \cite{Can91}. Moreover there is a relative Fourier-Mukai transform for most elliptic fibrations \cite{HMP02, HLS08} that can be
understood in terms of duality in string theory \cite{Don98,
DoPan03, DoPan12} or D-brane theory. 
More generally, due to the interpretation of B-type D-branes as objects of the derived category  \cite{Kon95} and  to
Kontsevich's  homological mirror symmetry proposal \cite{Kon95}, one expects the
Fourier-Mukai transform (or its relative version) to act on the spectrum of D-branes. The study of D-branes on
Calabi-Yau manifolds inspired in fact the search of new Fourier-
Mukai partners \cite{Or97, Kaw02a, Ue04, HKTY}  among other mathematical problems.
 
In this paper we will always work over an  algebraically closed  field $k$ of characteristic zero and all schemes are assumed to be separated noetherian and of finite type over $k$.

{\it Acknowledgements:} I thank Dar\'{\i}o S\'anchez G\'omez and Carlos Tejero Prieto for many useful discussions.

 \section{The proof of theorem 1.1}
 
 In the proof of our result, we will make use of the following facts.
  
\subsection{} Let $X$ and $Y$ be two proper schemes over $k$ and let $\cplx K$ be an object in $D^-(X\times Y)$. The {\it integral functor} defined by $\cplx{K}$ is the functor $\Phi^{\cplx K}_{X\to Y}\colon D^-(X)\to D^-(Y)$ given by 
$$\Phi_{X\to Y}^{\cplx K}(\cplx{E})=\bR \pi_{Y_\ast}( \pi_X^\ast\cplx{E}\lotimes\cplx{K})\, ,$$
where $\pi_X\colon X\times Y\to X$ and $\pi_Y\colon X\times Y\to Y$ are the two projections. 
The complex $\cplx K$ is said to be the kernel of the integral functor. 

\subsection{} Our theorem is just stated for  projective curves and the reason for it is that the projective context provides the following important advantage. Due to a famous result of Orlov \cite{Or97}, as extended by Ballard \cite{Ballard09} to the singular case, if $X$ and $Y$ are projective schemes over $k$ and $F\colon \cdbc{X}\simeq \cdbc{Y}$ is an equivalence of categories, then there is a unique (up to isomorphism) $\cplx{K}\in \cdbc{X\times Y}$ such that $F\simeq \Phi_{X\to Y}^{\cplx K}$, that is, any equivalence between the bounded derived categories of two projective schemes is a Fourier-Mukai functor. This result is now a consequence of a more general result by Lunts and Orlov \cite{OL10}.

 \subsection{} The existence of a Fourier-Mukai functor $\Phi_{X\to Y}^\cplx{K}\colon \cdbc{X}\simeq \cdbc{Y}$ between the derived categories of two projective schemes $X$ and $Y$ has important geometrical consequences. For instance, it is well know that if $X$ is a smooth scheme, then $Y$ is also smooth  (see for instance \cite{Bri98t}). In the same vein, Theorem 4.4 in \cite{HLS08} shows that if $X$ is an equidimensional Cohen-Macaulay  (resp. Gorenstein) scheme and $Y$ is reduced, then $Y$ is also equidimensional of the same dimension that $X$ and it is  Cohen-Macaulay (resp. Gorenstein) as well.
 
 Moreover, if $X$ is Gorenstein, Proposition 1.30 in \cite{HLS07} proves that $\omega_X$ is trivial if and only if $\omega_Y$ is trivial.

Furthermore,  one has the following result whose proof for smooth projective varieties was given by Orlov (Corollary 2.1.9 in \cite{Or003}).

\begin{prop}
Let  $X$ and $Y$ be two Gorenstein projective partners. Then, there is an  isomorphism  $$\oplus_{i\geq 0}H^0(X, \omega_X^i)\simeq \oplus_{i\geq 0}H^0(Y, \omega_Y^i)\, $$
between  the graded canonical algebras.
 \end{prop}
 
 \begin{proof} The proof in the  smooth case is based in the use of Serre's functors. For Gorenstein schemes, those functors only exist restricting to the subcategories of perfect complexes, then we proceed as follows.  By the representability theorem of Ballard-Orlov, there is a unique (up to isomorphism)  $\cplx{K}\in \cdbc{X\times Y}$, such that the equivalence $\cdbc{X}\simeq \cdbc{Y}$ is given by the  corresponding Fourier-Mukai functor $\fmf{\cplx{K}}{X}{Y}$. Using that both schemes have the same dimension and that, by Proposition 1.30 in \cite{HLS07}, the right and left adjoints to $\fmf{\cplx{K}}{X}{Y}$ are functorially isomorphic, we have an isomorphism of complexes
 \begin{equation}\label{e:dual} \cplx{L}:=\ \dcplx{K}\otimes \pi_X^\ast\omega_X\simeq\dcplx{K}\otimes
\pi_Y^\ast\omega_Y\, .
 \end{equation}
 If $\cplx{M}:=\cplx{L}\boxtimes\cplx{K}\in \cdbc{(X\times X)\times (Y\times Y)}$, Proposition 1.10 in \cite{LST10} proves that the functor $$\fmf{\cplx{M}}{X\times X}{Y\times Y}\colon \cdbc{X\times X}\iso \cdbc{Y\times Y}$$ is  an equivalence of categories. Let $\delta_X\colon X\hookrightarrow X\times X$ be the diagonal immersion of $X$ and for $n\in \mathbb{Z}$ consider $\cplx{H}:=\fmf{\cplx{M}}{X\times X}{Y\times Y}(\delta_{X\ast}\omega_X^n)$. Arguing exactly as in (\cite{Or003}, Proposition 2.16), one gets the isomorphism $$\fmf{\cplx{H}}{Y}{Y}\simeq \fmf{\cplx{K}}{X}{Y}\circ\fmf{\delta_{X\ast}\omega_X^n}{X}{X}\circ \fmf{\cplx{L}}{Y}{X}\, ,$$
and this  composition is isomorphic to $\fmf{\delta_{Y\ast}\omega_Y^n}{Y}{Y}$. Indeed, for any $\cplx{F}\in \cdbc{Y}$ one  has
\begin{align*}
 \fmf{\cplx{K}}{X}{Y}\circ\fmf{\delta_{X\ast}\omega_X^n}{X}{X}\circ \fmf{\cplx{L}}{Y}{X}(\cplx{F}) \simeq & 
\fmf{\cplx{K}}{X}{Y}(\bR\pi_{X_\ast}(\pi_Y^\ast \cplx{F}\lotimes\dcplx{K}\otimes \pi_X^\ast\omega_X^{n+1}))\\
\simeq &\fmf{\cplx{K}}{X}{Y}(\bR\pi_{X_\ast}(\pi_Y^\ast \cplx{F}\lotimes\dcplx{K}\otimes \pi_Y^\ast\omega_Y^{n+1}))\\
\simeq &\fmf{\cplx{K}}{X}{Y}\circ\fmf{\dcplx{K}\otimes\pi_Y^\ast\omega_Y}{Y}{X}(\cplx{F}\otimes \omega_Y^n)\\
\simeq & \cplx{F}\otimes \omega_Y^n\simeq \fmf{\delta_{Y\ast}\omega_Y^n}{Y}{Y}(\cplx{F})
\end{align*}
where the first isomorphism is projection formula for $\pi_X$, the second is obtained by equation \eqref{e:dual} and the forth is the adjunction.
 
 By the uniqueness of the kernel, we obtain  $$\cplx{H}\simeq \delta_{Y\ast}\omega_Y^n\, .$$
 Since $\fmf{\cplx{M}}{X\times X}{Y\times Y}$ is an equivalence, Parseval formula gives us isomorphisms 
 $$\Hom(\delta_{X\ast}\omega_X^i,\delta_{X\ast}\omega_X^j)\simeq \Hom(\delta_{Y\ast}\omega_Y^i, \delta_{Y\ast}\omega_Y^j)$$
for all $i,j\in \mathbb{Z}$.  Taking $i=0$, one finds
$$H^0(X,\omega_X^j)\simeq H^0(Y, \omega_Y^j)$$ for all $j\in\mathbb{Z}$. 
 Moreover, since $X$ is Gorenstein,  $\omega_X$ is a line bundle, and then $$\Hom(\mathcal{O}_X,\omega_X^j)\simeq\Hom (\omega_X^k,\omega_X^{k+j})$$    for all $k$ and similarly for $Y$.  Bearing in mind that the algebra structure is defined by composition, the induced bijection $$\oplus_{j} H^0(X,\omega_X^j)\simeq \oplus_{j}H^0(Y,\omega_Y^j))$$ is a ring isomorphism. \end{proof}

From this Proposition, we obtain in particular that if  $X$ and $Y$ are curves, they have the same arithmetic genus.

To finish this part, notice that the connectedness of a scheme  is also a property invariant under Fourier-Mukai functors, since a scheme $X$ is connected if and only if its bounded derived category $\cdbc{X}$ is indecomposable (see Example 3.2 in \cite{Bri99}).

 \subsection{} Let $K_\bullet(X):=
K(\cdbc{X})$ be the Grothendieck group of the triangulated category $\cdbc{X}$, that is,  the quotient of the free Abelian group generated by the objects of $\cdbc{X}$ modulo expressions coming from distinguished triangles.

To pass from $\cdbc{X}$ to $K_\bullet(X)$ one considers the map $[\quad]\colon \cdbc{X}\to K_\bullet(X)$ given by $\cplx{F}\mapsto [\cplx F]=\sum(-1)^i[\mathcal{H}^i(\cplx F)]$. Any integral functor $\Phi=\fmf{\cplx K}{X}{Y}\colon \cdbc{X}\to \cdbc{Y}$ induces a group morphism $\phi\colon K_\bullet(X)\to K_\bullet(Y)$ that commutes with the projections $[\quad]\colon \cdbc{X}\to K_\bullet(X)$, that is, such that the following diagram commutes 
$$
\xymatrix{
\cdbc{X}\ar[r]^\Phi\ar[d]_{[\quad]} &\cdbc{Y}\ar[d] ^{[\quad]} \\
K_\bullet(X)\ar[r]^{\phi}& K_\bullet(Y).}
$$
This morphism is given by $\phi(\alpha)={\pi_Y}_!\big(\pi_X^\ast\alpha\otimes [\cplx K]\big)$. Remember that a complex in $\cdbc{X}$ is called {\it perfect} if it is locally isomorphic to a bounded complex of locally free sheaves of finite rank and notice that $\phi$ is well defined if  $K_\bullet(X)$  is generated by perfect objects. Moreover if $\Phi$ is an equivalence then $\phi$ is an isomorphism.

 \subsection{}  Following Orlov, we will say that a scheme $X$ satisfies the (ELF) condition if  it is separated, noetherian, of finite Krull dimension and the category of coherent sheaves on $X$ has enough locally free sheaves. This is the case, for instance, of any quasi-projective scheme. The (ELF) condition on $X$ implies that any perfect complex is globally (and not only locally) isomorphic to a bounded complex of locally free sheaves of finite rank.  Perfect complexes form a full triangulated subcategory $\text{Perf}(X)\subseteq \cdbc{X}$ which is thick. For a scheme $X$ satisfying (ELF) condition, Orlov introduced in \cite{Or03} a new invariant, called {\it the category of singularities} of $X$, which is defined as the Verdier quotient triangulated category $$D_\text{sg}(X):=\cdbc{X}/\text{Perf}(X)$$

Denote by $q\colon \cdbc{X}\to D_\text{sg}(X)$ the natural localization functor.

One has the following evident lemma:

\begin{lem} Let $\mathcal{C}'$ and ${\mathcal D}'$ be full triangulated subcategories in triangulated categories $\mathcal{C}$ and $\mathcal{D}$. Let $F\colon \mathcal{C}\to \mathcal{D}$ and $G\colon \mathcal{D}\to \mathcal{C}$ be an adjoint pair of exact functors satisfying $F({\mathcal C}')\subset {\mathcal D}'$ and $G({\mathcal D}')\subset {\mathcal C}'$. Then they induce functors $$\overline{F}\colon {\mathcal C}/{\mathcal C}'\to {\mathcal D}/{\mathcal D}',\quad \overline{G}\colon {\mathcal D}/{\mathcal D}' \to {\mathcal C}/{\mathcal C}'$$ which are also an adjoint pair.  \qed
\end{lem}

\begin{cor} \label{Dsg} Let $X$ be a projective scheme over $k$. If $Y\in FM(X)$ is a projective partner, then $D_\text{sg}(X)\simeq D_\text{sg}(Y)$.
\end{cor}
\begin{proof} By the representability theorem of Ballard-Orlov, the equivalence $F\colon \cdbc{X}\simeq \cdbc{X}$ is a Fourier-Mukai functor, say $F\simeq \fmf{\cplx{K}}{X}{Y}$. By Proposition 2.10 in \cite{HLS08} we know that in the projective context, the kernel $\cplx{K}$ of an equivalence $\Phi_{X\to Y}^{\cplx K}\colon \cdbc{X}\simeq \cdbc{Y}$ has to be of finite homological dimension over both factors $X$ and $Y$. Being $\cplx{K}$ of finite homological dimension over the second factor $Y$,  Lemma 1.8 in \cite{LST10}  implies that $\Phi:=\Phi_{X\to Y}^{\cplx K}$ maps perfect complexes on $X$ to perfect complexes on $Y$, that is, $\Phi(\text{Perf}(X))\subseteq \text{Perf}(Y)$.
Since the quasi-inverse $\Psi=\Phi^{-1}$ is again an integral functor and its kernel  $\bR\dSHom{\cO_{X\times Y}}(\cplx
K,\pi_2^!\cO_Y)$ has also finite homological dimension over $X$ (see Proposition 2.10 in \cite{HLS08}), $\Psi(\text{Perf}(Y))\subseteq \text{Perf}(X)$. Using the above lemma,  one has then a pair of adjoint functors $$\overline{\Phi}\colon D_\text{sg}(X)\to D_\text{sg}(Y),\quad \overline{\Psi}\colon D_\text{sg}(Y)\to D_\text{sg}(X)$$ which are of course quasi-inverse.

\end{proof}

If $\mathcal{D}$ is a triangulated category and $\mathcal{S}$ is a collection of objects in $\mathcal{D}$, we denote $\langle S\rangle_\mathcal{D}$ the smallest thick subcategory of $\mathcal{D}$ containing $\mathcal{S}$. Unifying two results of Orlov, Chen proves in \cite{Chen10} the following result.

\begin{thm}\label{th:chen} Let $X$ be a (ELF) scheme over $k$ and $j\colon U\hookrightarrow X$ an open immersion. Denote by $Z=X/U$ the complement of $U$ and by $Coh_Z(X)\subseteq Coh(X)$ the subcategory of coherent sheaves on $X$ supported in $Z$. Then the induced functor $\bar{j^\ast}\colon  D_\text{sg}(X)\to  D_\text{sg}(U)$ induces a triangle equivalence $$ D_\text{sg}(X)/\langle q(Coh_ZX)\rangle\simeq  D_\text{sg}(U)\, .$$\qed
\end{thm}
\
He easily deduces the following two corollaries.

\begin{cor} If we denote $\text{Sing}(X)$ the singular locus of $X$, then
\begin{enumerate} \item  $\bar{j^\ast}\colon  D_\text{sg}(X)\to  D_\text{sg}(U)$  is an equivalence if and only if  $\text{Sing}(X)\subseteq U$.
\item $D_\text{sg}(X)=\langle q(Coh_ZX)\rangle$ if and only if  $\text{Sing}(X)\subseteq Z$.
\end{enumerate} \qed
\end{cor}

Let us denote $k(x)$ the skyscraper sheaf of a closed point $x\in X$, since every coherent sheaf supported in $\{x_1,\hdots,x_n\}$ belongs to $\langle k(x_1)\oplus\hdots\oplus k(x_n)\rangle_{D_\text{sg}(X)}$, one get the following

\begin{cor}\label{isolated} A scheme $X$ satisfying the (ELF) condition has only isolated singularities $\{x_1,\hdots,x_n\}$ if and only if 
 $$D_\text{sg}(X)\simeq\langle k(x_1)\oplus \hdots\oplus k(x_n)\rangle_{D_\text{sg}(X)}\, .$$ \qed
 \end{cor}

 \begin{cor} \label{c:isolated} Let $X$ be a projective scheme with $n$ isolated singularities. Then, any   projective Fourier-Mukai partner $Y\in FM(X)$  has also $n$ isolated singularities.\qed
 \end{cor}
 
 \subsection{} An additive category $\mathcal{A}$ is said to be {\it idempotent} if any idempotent morphism $e\colon a\to a$, $e^2=e$, arises from a splitting of $a$, $$a=\Img(e)\oplus\ker(e)\, .$$
 
 Let $\mathcal{A}$ be an additive category. {\it The idempotent completion} of $\mathcal{A}$ is the category $\overline{\mathcal{A}}$ defined as follows. Objects of $\overline{\mathcal{A}}$ are pairs $(a, e)$ where $a$ is an object of $\mathcal{A}$ and $e\colon a\to a$ is an idempotent morphism. A morphism in $\overline{\mathcal{A}}$ from $(a, e)$ to $(b,f)$ is a morphism $\alpha\colon  a\to b$ such that $\alpha\circ e=f\circ \alpha=\alpha$.
 
 There is a canonical fully faithful functor $i\colon \mathcal{A}\to \overline{\mathcal{A}}$ defined by sending $a\to (a,1_a)$.
 
 The following theorem can be found in \cite{BalSch01}.
 
 \begin{thm}\label{t:completion} Let $\mathcal{D}$ be a triangulated category. Then its idempotent completion $\overline{\mathcal{D}}$ admits a unique structure of triangulated category such that the canonical functor $i\colon \mathcal{D}\to \overline{\mathcal{D}}$ becomes an exact functor. If $\overline{\mathcal{D}}$ is endowed with this structure, then for each idempotent complete triangulated category $\mathcal{C}$ the functor $i$ induces an equivalence
$$\Hom(\overline{\mathcal{D}},\mathcal{C})\simeq \Hom(\mathcal{D},\mathcal{C})$$
where $\Hom$ denotes the category of exact functors.
 \end{thm}
\subsection{}The last tool we will  use  is the theory of maximal Cohen-Macaulay modules over a local ring. Let $(A,\mathfrak{m})$ be a local commutative  ring of Krull dimension $d$. Recall that  a finitely generated $A$-module $M$ is called a {\it maximal Cohen Macaulay} module if $\text{depth}(M)=d$ or equivalently $\text{Ext}^i(A/\mathfrak{m}, M)=0$ for all $i<d$. The category of all maximal Cohen-Macaulay modules over $A$ is denoted by $\CM(A)$. We will also denote $\sCM(A)$ the {\it stable category of Cohen-Macaulay modules} over $A$ which is defined as follows:
\begin{itemize}\item Ob$(\sCM(A))$=Ob($\CM(A)$).
\item $\Hom_{\sCM(A)}(M,N)=\Hom_{\CM(A)}(M,N)/\mathcal{P}(M,N)$ where $\mathcal{P}(M,N)$ is the submodule of $\Hom_{\CM(A)}(M,N)$ generated by those morphisms which factor through a free $A$-module.
\end{itemize}

 The following result is well-known to experts (see \cite{BK11}, \cite{Or09} or \cite{IyWe11}).

 \begin{thm}\label{t:equivalencia} Suppose that $X$ is  a Gorenstein scheme over $k$ satisfying the (ELF) condition and only  with isolated singularities $\{x_1,\hdots, x_n\}$. Then,
 there is a fully faithful functor $D_\text{sg}(X)\hookrightarrow \oplus_{i=1}^n \sCM(\mathcal{O}_{x_i})$ which is essentially surjective up to direct summands. Taking the idempotent completion gives an equivalence of categories $$\overline{D_\text{sg}(X)}\simeq \oplus_{i=1}^n\sCM(\hat{\mathcal{O}} _{x_i})$$ where $\hat{\mathcal{O}} _{X,x_i}$ is the completion of the local ring $\mathcal{O}_{X,x_i}$ at the point  $x_i$. \qed
 \end{thm}

\subsection{Proof of Theorem 1.1} Let $X$ be a projective reduced Gorenstein curve of arithmetic genus one and trivial dualising sheaf and let  $Y\in FM(X)$ be a projective Fourier-Mukai partner.
By the previous discussion, we already know that $Y$ is a Gorenstein projective curve of arithmetic genus one and trivial dualising sheaf. Since $X$ has only isolated singularities, the same is true for $Y$ which means that $Y$ is also reduced. Then, $Y$ belongs to Catanese's list. 
 
 If $X$ is a smooth elliptic curve, $Y$ has to be also smooth and we already know that it is isomorphic to $X$. 
  
The number of irreducible components for this kind of curves is also invariant under Fourier-Mukai transforms. Indeed, Proposition 2.3 in \cite{HLST09}, proves that if $Z$ is any reduced connected projective curve such that every irreducible component is isomorphic to a projective line, then $K_\bullet (Z)\simeq \mathbb{Z}^{N+1}$ where $N$ is the number of irreducible components. Since  one has a group isomorphism $K_\bullet(X)\simeq K_\bullet (Y)$ between the Grothendieck groups of both curves, $X$ and $Y$ have the same number of irreducible components. By Corollary \ref{c:isolated},  they also have the same number of singular points, and this allows to distinguish the case (1.5) from the case (1.4) with $N=2$. Hence,  taking into account Catanese's classification, the only thing to check is that a rational curve with one node cannot be a Fourier-Mukai partner of a rational curve with one cusp.

Let $X$ be a rational curve with one node $P$ and $Y$ a rational curve with one cusp $Q$ and suppose that $Y\in FM(X)$.  By Corollary 
\ref{Dsg}, they have equivalent categories of singularities $D_\text{sg}(X)\simeq D_{sg}(Y)$. By Theorem \ref{t:completion},  taking the idempotent completion gives an equivalence of categories $\overline{D_{sg}(X)}\simeq \overline{D_{sg}(Y)}$, and applying  now Theorem \ref{t:equivalencia}, we get an exact equivalence
\begin{equation} \sCM(R)\simeq \sCM(S)	\label{eqCM}
\end{equation}
between the stable categories of Cohen-Macaulay modules on  $R={\hat{\mathcal{O}}_{X,P}}$ and on $S:=\hat{\mathcal{O}}_{Y,Q}$. Notice that $$R\simeq k[[x,y]]/(x^2+y^2)$$ is a simple singularity of type $A_1$ while  $$S\simeq k[[x,y]]/(x^2+y^3)$$ is a simple singularity of type $A_2$. Both 
 are complete local rings containing a field, then analytic algebras with trivial valuation.
Since the equivalence \eqref{eqCM} is exact, it has to map indecomposable objects to indecomposable objects, so it establishes an one-to-one correspondence between the set of isomorphism classes of indecomposable Cohen-Macaulay $R$-modules and the set of isomorphism classes of indecomposable Cohen-Macaulay $S$-modules. In other words,  a one-to-one correspondence between the set  $v(\Gamma(R))$ of vertices of the AR-quiver (Auslander-Reiten-quiver) of $R$ and the set $v(\Gamma(S))$ of vertices of the AR-quiver of  $S$. 
But it is absurd because, as one can see in \cite{YCM90} where both quivers are computed, the set of  isomorphism classes of indecomposable Cohen-Macaulay $R$-modules is
$$v(\Gamma(R))=\{ [R], [N_{-}], [N_+]\}$$ where $N_{-} =R/(y-ix)$ and $N_+=R/(y+ix)$  while 
 $$v(\Gamma(S))=\{ [S], [I]\}$$ being $I=\langle x,y\rangle$. \qed
 
  \bibliographystyle{siam}

\begin{thebibliography}{10}

\bibitem{Ballard09}
{\sc R.~M. Ballard}, {\em Equivalences of derived categories of sheaves on
  quasi-projective schemes}.
\newblock math. AG. 09053148v2.

\bibitem{BalSch01}
{\sc P.~Balmer and M.~Schlichting}, {\em
  Idempotent completion of triangulated categories},
  J. Algebra, 236, (2001),  pp.~819--824.


\bibitem{BBH08}
{\sc C.~Bartocci, U.~Bruzzo, and D.~Hern{\'a}ndez~Ruip{\'e}rez}, {\em
  Fourier-{M}ukai and {N}ahm transforms in geometry and mathematical physics},
  vol.~276 of Progress in Mathematics, Birkh\"auser Boston Inc., Boston, MA,
  2009.
  
 
  
\bibitem{Bri98}
{\sc T.~Bridgeland}, {\em Fourier-{M}ukai transforms for elliptic surfaces}, J.
  Reine Angew. Math., 498 (1998), pp.~115--133.

\bibitem{Bri98t}
\leavevmode\vrule height 2pt depth -1.6pt width 23pt, {\em {F}ourier-{M}ukai
  Transforms for Surfaces and Moduli Spaces of Stable Sheaves}, PhD thesis,
  University of Edinburgh, 1998.

\bibitem{Bri99}
\leavevmode\vrule height 2pt depth -1.6pt width 23pt, {\em Equivalences of
  triangulated categories and {F}ourier-{M}ukai transforms}, Bull. London Math.
  Soc., 31 (1999), pp.~25--34.

\bibitem{BK11}
{\sc I.~Burban and M.~Kalck}, {\em Singularity category of a non-commutative
  resolution of singularities}.
\newblock arXiv:1103.3936.

\bibitem{BuKr05}
{\sc I.~Burban and B.~Kreu{\ss}ler}, {\em Derived categories of irreducible
  projective curves of arithmetic genus one}, Compos. Math., 142 (2006),
  pp.~1231--1262.

\bibitem{Can91}
{\sc P.~Candelas, X.~C. de~la Ossa, P.~S. Green, and L.~Parkes}, {\em A pair of
  {C}alabi-{Y}au manifolds as an exactly soluble superconformal theory},
  Nuclear Phys. B, 359 (1991), pp.~21--74.

\bibitem{Cata82}
{\sc F.~Catanese}, {\em Pluricanonical-{G}orenstein-curves}, in Enumerative
  geometry and classical algebraic geometry ({N}ice, 1981), vol.~24 of Progr.
  Math., Birkh\"auser Boston, Boston, MA, 1982, pp.~51--95.

\bibitem{Chen10}
{\sc X.-W. Chen}, {\em Unifying two results of {O}rlov on singularity
  categories}, Abh. Math. Semin. Univ. Hambg., 80 (2010), pp.~207--212.

\bibitem{ChuLo}
{\sc W.-Y. Chuang and J.~Lo}, {\em Stability and fourier-mukai transforms on on
  higher dimensional elliptic fibrations}.
\newblock {\tt arXiv:1307.1845}.

\bibitem{Don98}
{\sc R.~Donagi}, {\em Taniguchi lectures on principal bundles on elliptic
  fibrations}, in Integrable systems and algebraic geometry (Kobe/Kyoto, 1997),
  World Sci. Publishing, River Edge, NJ, 1998, pp.~33--46.

\bibitem{DoPan03}
{\sc R.~Donagi and T.~Pantev}, {\em Torus fibrations, gerbes, and duality},
  Mem. Amer. Math. Soc., 193, No. 901 (2008).

\bibitem{DoPan12}
{\sc R.~Donagi and T.~Pantev}, {\em Langlands duality for {H}itchin systems},
  Invent. Math., 189 (2012), pp.~653--735.

\bibitem{HLST09}
{\sc D.~Hern{\'a}ndez~Ruip{\'e}rez, A.~C. L{\'o}pez~Mart{\'{\i}}n,
  D.~S{\'a}nchez~G{\'o}mez, and C.~Tejero~Prieto}, {\em Moduli spaces of
  semistable sheaves on singular genus 1 curves}, Int. Math. Res. Not. IMRN,
  (2009), pp.~4428--4462.

\bibitem{HLS07}
{\sc D.~Hern\'andez~Ruip\'erez, A.~C. L\'opez~Mart\'{\i}n, and F.~Sancho~de
  Salas}, {\em {F}ourier-{M}ukai transform for {G}orenstein schemes}, Adv.
  Math., 211 (2007), pp.~594--620.

\bibitem{HLS08}
{\sc D.~Hern{\'a}ndez~Ruip{\'e}rez, A.~C. L{\'o}pez~Mart{\'{\i}}n, and
  F.~Sancho~de Salas}, {\em Relative integral functors for singular fibrations
  and singular partners}, J. Eur. Math. Soc. (JEMS), 11 (2009), pp.~597--625.

\bibitem{HMP02}
{\sc D.~Hern{\'a}ndez~Ruip{\'e}rez and J.~M. Mu{\~n}oz~Porras}, {\em Stable
  sheaves on elliptic fibrations}, J. Geom. Phys., 43 (2002), pp.~163--183.

\bibitem{HiVdB}
{\sc L.~Hille and M.~D. Van~den Bergh}, {\em {F}ourier-{M}ukai transforms}, in
  {H}andbook on tilting theory, vol.~332 of London Math. Soc. Lecture Note
  Series, Cambridge Univ. Press, 2007.

\bibitem{HKTY}
{\sc S.~Hosono, A.~Klemm, S.~Theisen, and S.-T. Yau}, {\em Mirror symmetry,
  mirror map and applications to complete intersection {C}alabi-{Y}au spaces},
  Nuclear Phys. B, 433 (1995), pp.~501--552.

\bibitem{IyWe11}
{\sc O.~Iyama and M.~Wemyss}, {\em Singular derived categories of
  $\mathbb{Q}$-factorial terminalizations and maximal modification algebras}.
\newblock arXiv:1108.4518.

\bibitem{Kaw02a}
{\sc Y.~Kawamata}, {\em {$D$}-equivalence and {$K$}-equivalence}, J.
  Differential Geom., 61 (2002), pp.~147--171.

\bibitem{Kon95}
{\sc M.~Kontsevich}, {\em Homological algebra of mirror symmetry}, in
  Proceedings of the International Congress of Mathematicians, Vol.\ 1, 2
  (Z\"urich, 1994), Basel, 1995, Birkh\"auser, pp.~120--139.

\bibitem{Lo}
{\sc J.~Lo}, {\em Stability and fourier-mukai transforms on elliptic
  fibrations}.
\newblock {\tt arXiv:1206.4281v2}.

\bibitem{LM05a}
{\sc A.~C. L{\'o}pez~Mart{\'{\i}}n}, {\em Relative {J}acobians of elliptic
  fibrations with reducible fibers}, J. Geom. Phys., 56 (2006), pp.~375--385.

\bibitem{LST10}
{\sc A.~C. L{\'o}pez~Mart{\'{\i}}n, D.~S\'anchez~G\'omez, and
  C.~Tejero~Prieto}, {\em Relative fourier-mukai functors for {W}eierstra{\ss}
  fibrations, abelian schemes and fano fibrations}, doi:
  10.1017/S0305004113000029, Math. Proc. Camb. Phil. Soc., 00 (2013),
  pp.~1--25.
  
  \bibitem{OL10}
{\sc V.~Lunts, D.~Orlov}, {\em Uniqueness of enhancement for triangulated categories}, J. Amer. Math. Soc., 23, 3 (2010),
  pp.~853--908.


\bibitem{Or09}
{\sc D.~Orlov}, {\em Formal completions and idempotent completions of
  triangulated categories of singularities}, Adv. Math., 226 (2011),
  pp.~206--217.

\bibitem{Or97}
{\sc D.~O. Orlov}, {\em Equivalences of derived categories and ${K}3$
  surfaces}, J. Math. Sci. (New York), 84 (1997), pp.~1361--1381.
\newblock Algebraic geometry, 7.

\bibitem{Or003}
\leavevmode\vrule height 2pt depth -1.6pt width 23pt, {\em Derived categories
  of coherent sheaves and equivalences between them}, Uspekhi Mat. Nauk, 58
  (2003), pp.~89--172.

\bibitem{Or03}
\leavevmode\vrule height 2pt depth -1.6pt width 23pt, {\em Triangulated
  categories of singularities and {D}-branes in {L}andau-{G}inzburg models},
  Tr. Mat. Inst. Steklova, 246 (2004), pp.~240--262.


\bibitem{Ue04}
{\sc H.~Uehara}, {\em An example of {F}ourier-{M}ukai partners of minimal
  elliptic surfaces}, Math. Res. Lett., 11 (2004), pp.~371--375.

\bibitem{YCM90}
{\sc Y.~Yoshino}, {\em Cohen-{M}acaulay modules over {C}ohen-{M}acaulay rings},
  vol.~146 of London Mathematical Society Lecture Note Series, Cambridge
  University Press, Cambridge, 1990.

\end{thebibliography}

\def\cprime{$'$}

  \end{document}